\newif\ifArXiV

\ArXiVtrue 

\ifArXiV
\documentclass{article}
\else
\documentclass[authoryear,preprint]{elsarticle}
\fi

\usepackage[utf8x]{inputenc}
\usepackage[T1]{fontenc}
\usepackage{fullpage}
\usepackage[english]{babel}
\usepackage{amsmath,amssymb,amsthm,graphicx}
\usepackage{multirow}
\usepackage{hyperref}
\usepackage{subcaption}
\usepackage{booktabs}
\usepackage{xspace}
\usepackage{graphicx}

\newtheorem{theorem}{Theorem}
\newtheorem{definition}[theorem]{Definition}

\newcommand{\MCLP}{MCLP\xspace}
\newcommand{\MCLPCP}{MCLPCP\xspace}
\newcommand{\BMCLP}{BMCLP\xspace}
\newcommand{\BMCLPCP}{BMCLPCP\xspace}

\ifArXiV
\usepackage[affil-it]{authblk}
\usepackage[authoryear]{natbib}
\newenvironment{frontmatter}{}{}
\newenvironment{keyword}{\small \textbf{Keywords:}}{}
\let\address\affil
\fi

\begin{document}
\begin{frontmatter}
\title{A note on the maximal covering location problem with customer 
preference ordering}

\ifArXiV
\author[1]{Elisabeth Gaar\thanks{elisabeth.gaar@uni-a.de}}
\author[2]{Markus Sinnl\thanks{markus.sinnl@jku.at}}
\affil[1]{Institute of Mathematics, University of Augsburg, Augsburg, Germany}
\affil[2]{Institute of Business Analytics and Technology Transformation/JKU Business School, Johannes Kepler University Linz, Linz, Austria}			
\date{}
\maketitle

\fi

\begin{abstract}

Recently a series of papers introduced and investigated the maximal 
covering location problem with customer preference ordering, a variant of the classical maximal covering location problem (\MCLP).
In these papers, mixed-integer bilevel optimization models and single-level 
reformulations were presented for this problem, as well as various heuristics such 
as a GRASP, a Tabu search and a variable neighborhood search.

In this short note we show that instances of this new problem can actually be easily 
transformed into instances of the classical \MCLP and this transformation even 
reduces the size of the instance. Thus, existing algorithms for the classical 
\MCLP can be used to solve it. We provide a short computational study to show that this transformation leads to speed-ups of at least a magnitude when considering exact algorithms.

\end{abstract}
\end{frontmatter}

\section{Main result}\label{sec:main}
The \emph{maximal covering location problem (\MCLP)}, originally introduced by 
\citet{church1974maximal} is a fundamental problem in location science. In this 
problem, we are given a set of customers with demands and a set of potential 
facility locations. Each 
facility location can cover a certain subset of customers. Moreover, we are given 
a cardinality constraint, i.e., a number $p$, and the goal is to open $p$ 
facilities in order to maximize the sum of the demands of the customers which are 
covered by the open facilities. 

There also exists a version of the \MCLP with a budget constraint instead of the cardinality constraint, where a budget $\mathcal B$ is given, and the sum of the costs of the opened facilities are not allowed to exceed $\mathcal B$. A formal definition of the this problem is given as follows.

\begin{definition}[Budget-constrained maximal covering location problem (\BMCLP)]
	Let $\mathcal I$ be the set of potential facility locations, and $\mathcal J$ 
	be the set of customers with demands $\mathcal D_j$, $j \in \mathcal J$. Let 
	$\mathcal 
	I(j) \subseteq \mathcal I$ represent the set of facility locations which can 
	cover customer $j$. For each facility location $i \in \mathcal I$, the cost $d_i$ of
	opening facility $i$ is given. Let $\mathcal B$ be the budget for 
	opening facilities at some of the locations from $\mathcal I$. The goal 
	is to open some of the facilities of $\mathcal I$ respecting the budget $\mathcal B$ in such a way that the sum of the demands of 
	customers covered by the opened facilities is maximized.
\end{definition}

It is easy to see that the \MCLP can be obtained as a special case of the \BMCLP where the costs $d_i = 1$ for all $i \in \mathcal{I}$ and $\mathcal{B} = p$. Thus, we restrict ourselves to the study of the \BMCLP from now on.

Naturally, over the years many extensions and variants of the \MCLP and the \BMCLP have emerged, 
see e.g., Chapter~5 of \citet{laporte2019location} for more details on covering 
location problems. In this work, 
we focus on the 
so-called \emph{maximal covering location problem with customer
	preference ordering (\MCLPCP)}
which was recently introduced by \cite{diaz2017grasp} and also studied in 
\citet{mrkela2018vns} and \citet{casas2020bi} and which has a cardinality constraint. Moreover, a 
variant with a budget constraint was introduced 
in \citet{mrkela2022variable}. It is defined as follows.

\begin{definition}[Budget‑constrained
	maximal covering location problem with customer
	preference ordering (\BMCLPCP)]
Let $I_1$ be the set of potential locations to open facilities of the company 
entering
the market and $I_2$ be the set of existing facility locations belonging to 
competitors,
while $I = I_1 \cup I_2$. 
For each facility location $i \in I_1$, the cost $c_i$ of
opening this facility is given.

The set of customers is $J$.
A demand $D_j$ is assigned to each customer $j \in J$, while $I(j) \subseteq I$
represents the subset of facility locations that can cover customer $j$.

The preference of customer $j \in J$ toward a facility at $i \in I$ is given by
$g_{ij}$. If $g_{i_1 j} > g_{i_2 j}$ holds, it means that customer $j$ prefers to 
be
allocated to a facility on site $i_1$ over a facility on site $i_2$. We assume that the preferences of a customer $j$ induce a strict total order of the facility locations that cover $j$, i.e., $g_{i_1 j} \neq g_{i_2 j}$ for all $i_1 \neq i_2 \in I(j)$.
The company that plans to enter the market has a limited budget $B$ for opening
facilities at some of the locations from the set~$I_1$.	

The goal of the company entering the market is to open 
some  
facilities $S \subseteq I_1$ respecting the budget $B$ in such a way that the sum of the demands of 
customers allocated to $S$ is maximized, where a customer $j$ is 
only \emph{allocated} to an open facility $i$ from $S$ if $i\in I(j)$ (i.e., $j$ is covered by $i$),
$g_{ij}\geq g_{i'j}$ for 
all $i' \in S \cap I(j)$ (i.e., $i$ is the most preferred facility that covers $j$ within $S$) 
and $g_{ij}>g_{i'j}$ for 
all $i' \in I_2 \cap I(j)$ (i.e., $j$ is not covered by a more preferred facility $i'$ in $I_2$).
\end{definition}

Again, this budget-constraint version is a generalization of the cardinality-constraint version of the problem, the \MCLPCP. In the \MCLPCP a maximum cardinality $p$ of facilities to open is given instead of $c_{i}$ and~$B$, so it is a special case of the \BMCLPCP with $c_i = 1$ for all $i \in I$ and $B = p$.

In the works \citet{diaz2017grasp,mrkela2018vns,casas2020bi,mrkela2022variable} 
various exact and heuristic solution algorithms were proposed for   
the \MCLPCP and the \BMCLPCP. 
The proposed heuristics  
include a greedy randomized adaptive search procedure (GRASP), a Tabu search and a 
variable neighborhood search (VNS). Instances 
with up to 1000 potential facility 
locations and 9000 customers are tackled with these heuristics.

The exact approaches in the above mentioned works are based on 
mixed-integer programming (MIP), i.e., various MIP formulations for the \BMCLPCP 
were presented, including bilevel MIPs. The computationally most efficient 
formulation uses a variant of the so-called \emph{closest assignment constraints} 
(CAC), see, e.g., \cite{espejo2012closest} for a general discussion of these constraints, 
and is based 
on a non-standard way of formulating the \MCLP using assignment variables and 
facility opening variables: Let binary variables $y_i$, $i \in I$ indicate if facility $i$ is opened in a solution and binary variables $x_{ij}$ indicate that customer $j \in J$ is allocated to facility $i \in I(j)$. Let $J(i)=\{j \in J: i \in I(j)\}$, i.e., the set of all customers which can be covered by a facility $i \in I$. Let $J_2=\{j \in J: I(j)\cap I_2 \neq \emptyset \}$ be the set of customers that are covered by an already existing facility and let $J_1=J \setminus J_2$. Moreover, let $i_{j,1}, \dots, i_{j,|I(j)|} \in I(j)$ be such that $I(j) = \{i_{j,1}, i_{j,2}, \dots, i_{j,|I(j)|} \}$ and such that $g_{i_{j,1},j} > g_{i_{j,2},j} > \dots > g_{i_{j,|I(j)|},j}$, i.e., such that $i_{j,\cdot}$ induce an ordering of $I(j)$ according to preferences.
Then the formulation (see \citet{diaz2017grasp,mrkela2022variable}) reads as 
{\allowdisplaybreaks
\begin{subequations}
\label{f1}
\begin{align}
\max\ & \sum_{i \in I_1} \sum_{j \in J(i)} D_j x_{ij} \label{f1:start}\\
\text{s.t.}\quad
& y_i = 1 && \forall i \in I_2 \label{f1:comp}\\
& \sum_{i \in I_1} c_i y_i \leq B && \label{f1:budget} \\
& \sum_{i \in I(j)} x_{ij} \le 1 && \forall j \in J_1 \label{f1:cov1} \\
& \sum_{i \in I(j)} x_{ij} = 1 && \forall j \in J_2  \label{f1:cov2}  \\
& x_{ij} \le y_i && \forall i \in I,\ \forall j \in J(i) \label{f1:link} \\
& \sum_{s=k+1}^{|I(j)|} x_{\,i_{j,s},\,j} + y_{\,i_{j,k}} \le 1 && \forall j \in J,\ \forall k \in \{ 1,\dots, |I(j)|-1 \} \label{f1:cac}\\
& x_{ij} \in \{0,1\} && \forall i \in I,\ \forall j \in J(i)  \\
& y_i \in \{0,1\} && \forall i \in I. \label{f1:end}
\end{align}
\end{subequations}
}

The objective function \eqref{f1:start} maximizes the sum of the demands of the customers covered by facilities of the company entering the market. Constraints \eqref{f1:comp} ensure that all existing facilities belonging to competitors are open. Constraint \eqref{f1:budget} represents the budget-constraint of the company entering the market. Constraints \eqref{f1:cov1} and \eqref{f1:cov2} make sure that each customer is covered at most once. Note that \eqref{f1:cov2} can be written with equality, because by definition of $J_2$ the customers appearing in \eqref{f1:cov2} are covered by a facility in $I_2$ and these facilities are open. The linking constraints \eqref{f1:link} ensure that customers can only be covered by an open facility. Finally, the CAC \eqref{f1:cac} deal with the customer preference ordering, i.e., if facility $i_{j,k}$ is opened, customer $j$ cannot be assigned to any facility which is less preferred than $i_{j,k}$ by $j$.

Clearly this formulation has $O(|I||J|)$ variables and 
constraints. Using it, \BMCLPCP-instances with up to $|I| = 150$ potential facility 
locations and $|J| = 1300$
customers could be solved to optimality in \citet{mrkela2022variable} within a time limit of one hour. 

In contrast to this formulation for the \BMCLPCP, 
the standard MIP-formulation of the \BMCLP uses facility 
opening variables (i.e., $y_i$ as above) and customer variables (i.e., binary variables $z_j$, which are one if and only if a customer $j$ is covered in a solution), and reads as follows (see, e.g., \citet{church1974maximal,cordeau2019benders}):
\begin{subequations}
\label{f2}
\begin{align}
\max \quad     & \sum_{j \in \mathcal J} \mathcal D_j \, z_j \label{f2:start}\\
\text{s.t.}\quad
& z_j \;\le\; \sum_{i \in \mathcal I(j)} y_i \quad && \forall j \in \mathcal J \\
& \sum_{i \in \mathcal I} d_i \, y_i \;\le\; \mathcal B \\
& y_i \in \{0,1\} \quad && \forall i \in \mathcal I \\
& z_j \in \{0,1\} \quad && \forall j \in \mathcal J. \label{f2:end}
\end{align}
\end{subequations}
This formulation has $O(|\mathcal I|+|\mathcal J|)$ variables and 
$O(|\mathcal I|)$ constraints. Moreover, for  
the \BMCLP there also exist 
formulations based on 
Benders decomposition (see, e.g., \citet{cordeau2019benders,guney2021large}), 
which allow the 
exact solution of the \BMCLP for instances with 1000s of potential facility 
locations and millions of 
customers. Moreover, there are countless heuristics for the problem in the literature, see, e.g., 
\cite{maximo2017intelligent} and the references therein. Thus, it could be extremely
beneficial from a 
computational point-of-view to be 
able to re-use existing approaches and formulations for the \BMCLP also for the 
\BMCLPCP.

As a step toward this direction, we are now ready to state our main result.

\begin{theorem}\label{thm:main}
The optimal solution of an instance of the \BMCLPCP can be found by solving an instance of the \BMCLP. Consequently, also the optimal solution of an instance of the \MCLPCP can be obtained by solving an instance of the \MCLP.
\end{theorem}

\begin{proof}
Given an instance $P=(I_1,I_2,J,D_j,I(j),c_i,g_{ij},B)$ of the 
\BMCLPCP, we define an instance of the \BMCLP $\mathcal{P}=(\mathcal I,\mathcal J, \mathcal D_j,\mathcal I(j), d_i,\mathcal B)$  as follows: Set 
$\mathcal I=I_1$, 
$\mathcal J=J$, $\mathcal D_j=D_j$, $d_i = c_i$ and $\mathcal B=B$. Furthermore, for all customers $j \in \mathcal J$, 
define 
$\mathcal 
I(j)=\{i \in I_1 \cap I(j):g_{ij} > \max_{i' \in  I_2 \cap I(j) } g_{i'j}\}$. 

Let $S^* \subseteq I_1$. Then $S^*$ can be seen as a set of opened facilities and therefore represents a solution of both~$P$ and $\mathcal{P}$. Clearly $S^*$ is feasible for $P$ if and only if $S^*$ is feasible for $\mathcal P$ trivially holds, as $\mathcal I=I_1$, $\mathcal B=B$ and $d_i=c_i$ for all $i \in \mathcal I$. 
We next show that the objective function values of the solution $S^*$ for $P$ and for $\mathcal{P}$ coincide. 
To do so, we observe that it is enough to prove that the customers allocated to the opened facilities in $S^*$ in $P$ are exactly the same as the customers covered by the opened facilities in $S^*$ in $\mathcal{P}$.

For the first direction, let $J(S^*)$ be the set of customers allocated to the opened facilities in $S^*$ in $P$. 
Thus, for each $j \in J(S^*)$ it must hold that there exists an open facility
$i^*(j) \in S^*$ such that  $i^*(j) \in I_1 \cap I(j)$ and $g_{i^*(j)j} > 
\max_{i' \in I_2 \cap I(j) } g_{i'j}$, as otherwise $j$ would not be 
allocated (as a facility in $I_2$ which can also cover it would be preferred). 
This condition is exactly the 
definition of $\mathcal I(j)$, thus $S^*$ also covers $j$ in $\mathcal{P}$.

For the other direction, let $\mathcal J( S^*)$ be the set of customers 
covered 
by the opened facilities in $ S^*$ in $\mathcal{P}$. By the definition of $\mathcal I(j)$, for each $j \in \mathcal 
J( S^*)$ it must hold that there exists 
an $i^*(j) \in  S^*$ such that  $i^*(j) \in I_1 \cap I(j) $ and 
$g_{i^*(j)j} > 
\max_{i' \in  I_2 \cap I(j) } g_{i'j}$. 
Let $I^*(j)$ be the set of all the potential choices for $i^*(j)$ in such a way.
Let $i^{**}(j) \in I^*(j)$ be the one choice for which $g_{i^{**}(j)j}\geq g_{i^*(j)j}$ holds
for all $i^*(j) \in I^*(j)$.    
Thus, this $i^{**}(j)$ also ensures 
that $j$ is allocated to $S^*$ in $P$, as there is no facility in both $I_1$ and $I_2$ which 
can also cover $j$ and would be preferred to $i^{**}(j)$.

As a consequence, the customers allocated to $S^*$ in $P$ are exactly the ones covered by $S^*$ in $\mathcal{P}$. Thus, the optimal objective function values of solutions $S^*$ for $P$ and $\mathcal{P}$ coincide. This shows that any solution is optimal for $P$ if and only if it is optimal for $\mathcal{P}$, which finishes the proof for the \BMCLPCP. The result for the \MCLPCP follows from the fact that it is a special case of the \BMCLPCP.
\end{proof}

As a consequence of Theorem~\ref{thm:main}, existing algorithms for the 
\BMCLP can be used to solve the \BMCLPCP. We note that the instance obtained for 
the \BMCLP as described in the proof of Theorem~\ref{thm:main} is smaller than the instance of the \BMCLPCP, as the set $I_2$ is not 
needed and also $|\mathcal I(j)|\leq |I(j)|$ holds for all $j \in \mathcal J=J$. 
In particular, the idea behind our transformation of an instance of the \BMCLPCP into an instance of the \BMCLP is that a customer will be allocated to a newly open facility in $I_1$, if and only if there is no more preferred facility of the competitor in $I_2$, since all existing facilities in $I_2$ are open for sure. Thus, the decision about the allocation in the \BMCLPCP can be modeled as the decision of the covering in the \BMCLP. 

Note that the \BMCLP does not explicitly consider and thus does also not provide allocation. However, it is easy to see that for any set of open facilities, including an optimal solution, an allocation to the most preferred open facility for any customer can be trivially obtained by inspection.

Thus, by Theorem~\ref{thm:main} any instance of the \BMCLPCP can be solved by solving a smaller instance of the \BMCLP. In the next section, we provide a short computational study to illustrate the benefits of this transformation.

\section{Computational study}
In our computational study, we compare the performance obtained when solving \BMCLPCP-instances directly via the best formulation from literature, namely~\eqref{f1}, against transforming the \BMCLPCP-instances into \BMCLP-instances via Theorem \ref{thm:main} and then solving the resulting \BMCLP-instances using formulation~\eqref{f2}. The implementation was done in C++ using CPLEX 20.1 as MIP solver. All settings of CPLEX were left on their default values. The computations were performed on a single core of an AMD Ryzen 5 2600 with~16~GB of RAM, and the time limit for each run was set to 300~seconds.

\subsection{Instances}
As the original instances from the computational studies in \citet{diaz2017grasp,mrkela2022variable} were not available, we generated instances following the description in \citet{mrkela2022variable} (which is similar to the one in \citet{diaz2017grasp} and just adds the budget constraint): In a first step $|I|+|J|$ points are generated in the unit square uniformly at random, and $|I|$ of them are taken as potential facility location, while the remaining ones are the customers. The set $|I_2|$ is determined by randomly taking~$10\%$ of~$I$. To obtain the sets $I(j)$ for $j \in J$, Euclidean distances between each $i \in I$ and $j \in J$ are computed, and a fixed coverage radius $R$ is used to determine the set, i.e., if the distance between some $i$ and $j$ is at most~$R$, then $i \in I(j)$. For each $i \in I$, $c_i$ is a random value from the interval $[1000,1500]$.
As budget, the value of $B=1000p$ is used, where the value of $p$ varies over the instances.
The user preferences are generated following the procedure introduced by \citet{canovas2007strengthened} in a work about the facility location problem with user preferences. As none of \citet{diaz2017grasp,mrkela2022variable} described how $D_j$ was determined, we set it to one for each customer. We note that in both formulations, the demands $D_j$ appear only as coefficients in the objective function, thus the effect of different values of $D_j$ than one on the computational performance should be the same for both formulations. 
Following \citet{mrkela2022variable}, we generated 60 instances, in particular ten instances for each of the following combinations of $(|I|,|J|,R,p)$: $(25,225,0.8,3)$, $(50,450,0.7,6)$, $(75,675,0.5,10)$, $(100,900,0.3,13)$, $(150,1300,0.25,20)$, $(200,1800,0.2,27)$.

\subsection{Results}

Tables \ref{tab:comparison1} and \ref{tab:comparison2} show detailed results of the computations, giving for each instance and solution approach the value of the best found solution (column $best$), the upper bound (column $UB$), the runtime in seconds (column $t (s)$, where TL indicates that the time limit of 300 seconds was reached). 
The runtimes show that for all instances, the transformation-approach is much faster. 
In particular, for all instances that could be solved within the time limit, the average speed-up factor is about 22.
For the instances with $|I|=150$ potential facility locations, none of them could be solved within the time limit using \eqref{f1}, while using the transformation all instances could be solved.

\begin{table}[h!tb]
\small
\centering
\caption{Comparison of results for smaller instances}
\label{tab:comparison1}
\begin{tabular}{rrrrrrrrr}
\toprule
  &  &  & \multicolumn{3}{c}{\eqref{f1}} & \multicolumn{3}{c}{\eqref{f2}} \\
\cmidrule(lr){4-6} \cmidrule(lr){7-9}
$|I|$ & $|J|$ & ID & $best$ & $UB$ & $t (s)$ & $best$ & $UB$ & $t (s)$ \\
\midrule
25 & 225 & 1 & 108 & 108 & 1.0 & 108 & 108 & \textbf{0.1} \\
25 & 225 & 2 & 109 & 109 & 0.3 & 109 & 109 & \textbf{0.0} \\
25 & 225 & 3 & 117 & 117 & 2.0 & 117 & 117 & \textbf{0.1} \\
25 & 225 & 4 & 103 & 103 & 1.6 & 103 & 103 & \textbf{0.1} \\
25 & 225 & 5 & 114 & 114 & 0.3 & 114 & 114 & \textbf{0.0} \\
25 & 225 & 6 & 132 & 132 & 2.6 & 132 & 132 & \textbf{0.1} \\
25 & 225 & 7 & 118 & 118 & 0.6 & 118 & 118 & \textbf{0.0} \\
25 & 225 & 8 & 135 & 135 & 1.9 & 135 & 135 & \textbf{0.0} \\
25 & 225 & 9 & 138 & 138 & 1.0 & 138 & 138 & \textbf{0.0} \\
25 & 225 & 10 & 115 & 115 & 1.6 & 115 & 115 & \textbf{0.1} \\ \hline
50 & 450 & 1 & 256 & 256 & 5.2 & 256 & 256 & \textbf{0.3} \\
50 & 450 & 2 & 292 & 292 & 5.8 & 292 & 292 & \textbf{0.2} \\
50 & 450 & 3 & 260 & 260 & 6.3 & 260 & 260 & \textbf{0.3} \\
50 & 450 & 4 & 273 & 273 & 7.0 & 273 & 273 & \textbf{0.3} \\
50 & 450 & 5 & 270 & 270 & 6.2 & 270 & 270 & \textbf{0.2} \\
50 & 450 & 6 & 266 & 266 & 6.0 & 266 & 266 & \textbf{0.3} \\
50 & 450 & 7 & 289 & 289 & 5.1 & 289 & 289 & \textbf{0.2} \\
50 & 450 & 8 & 266 & 266 & 4.8 & 266 & 266 & \textbf{0.2} \\
50 & 450 & 9 & 235 & 235 & 4.4 & 235 & 235 & \textbf{0.2} \\ 
50 & 450 & 10 & 311 & 311 & 10.7 & 311 & 311 & \textbf{0.3} \\ \hline
75 & 675 & 1 & 449 & 449 & 22.4 & 449 & 449 & \textbf{0.8} \\
75 & 675 & 2 & 431 & 431 & 38.5 & 431 & 431 & \textbf{1.5} \\
75 & 675 & 3 & 433 & 433 & 26.1 & 433 & 433 & \textbf{1.4} \\
75 & 675 & 4 & 403 & 403 & 32.3 & 403 & 403 & \textbf{2.1} \\
75 & 675 & 5 & 428 & 428 & 22.1 & 428 & 428 & \textbf{1.2} \\
75 & 675 & 6 & 405 & 405 & 17.8 & 405 & 405 & \textbf{0.9} \\
75 & 675 & 7 & 404 & 404 & 30.5 & 404 & 404 & \textbf{1.6} \\
75 & 675 & 8 & 370 & 370 & 16.7 & 370 & 370 & \textbf{1.4} \\
75 & 675 & 9 & 397 & 397 & 19.8 & 397 & 397 & \textbf{1.3} \\
75 & 675 & 10 & 412 & 412 & 30.0 & 412 & 412 & \textbf{1.4} \\
\bottomrule
\end{tabular}
\end{table}

\begin{table}[h!tb]
\small
\centering
\caption{Comparison of results for larger instances}
\label{tab:comparison2}
\begin{tabular}{rrrrrrrrrr}
\toprule
  &  &  & \multicolumn{3}{c}{\eqref{f1}} & \multicolumn{3}{c}{\eqref{f2}} \\
\cmidrule(lr){4-6} \cmidrule(lr){7-9}
$|I|$ & $|J|$ & ID & $best$ & $UB$ & $t (s)$ & $best$ & $UB$ & $t (s)$ \\
\midrule
100 & 900 & 1 & 519 & 519 & 108.2 & 519 & 519 & \textbf{9.0} \\
100 & 900 & 2 & 548 & 548 & 50.7 & 548 & 548 & \textbf{2.5} \\
100 & 900 & 3 & 580 & 586 & TL & 581 & 581 & \textbf{5.8} \\
100 & 900 & 4 & 575 & 575 & 129.0 & 575 & 575 & \textbf{3.1} \\
100 & 900 & 5 & 572 & 572 & 84.0 & 572 & 572 & \textbf{4.3} \\
100 & 900 & 6 & 585 & 585 & 84.0 & 585 & 585 & \textbf{3.3} \\
100 & 900 & 7 & 549 & 549 & 60.5 & 549 & 549 & \textbf{3.3} \\
100 & 900 & 8 & 514 & 514 & 162.9 & 514 & 514 & \textbf{10.1} \\
100 & 900 & 9 & 566 & 566 & 51.4 & 566 & 566 & \textbf{1.0} \\
100 & 900 & 10 & 537 & 537 & 269.0 & 537 & 537 & \textbf{11.8} \\ \hline
150 & 1300 & 1 & 819 & 847 & TL & 828 & 828 & \textbf{66.0} \\
150 & 1300 & 2 & 891 & 913 & TL & 896 & 896 & \textbf{54.9} \\
150 & 1300 & 3 & 881 & 893 & TL & 883 & 883 & \textbf{20.3} \\
150 & 1300 & 4 & 828 & 850 & TL & 832 & 832 & \textbf{61.3} \\
150 & 1300 & 5 & 828 & 862 & TL & 840 & 840 & \textbf{69.0} \\
150 & 1300 & 6 & 810 & 846 & TL & 821 & 821 & \textbf{225.2} \\
150 & 1300 & 7 & 840 & 883 & TL & 855 & 855 & \textbf{169.7} \\
150 & 1300 & 8 & 813 & 842 & TL & 822 & 822 & \textbf{91.2} \\
150 & 1300 & 9 & 822 & 857 & TL & 834 & 834 & \textbf{119.6} \\
150 & 1300 & 10 & 786 & 812 & TL & 795 & 795 & \textbf{74.3} \\ \hline
200 & 1800 & 1 & 1174 & 1220 & TL & \textbf{1190} & \textbf{1200} & TL \\
200 & 1800 & 2 & 1144 & 1193 & TL & 1165 & 1165 & \textbf{220.3} \\
200 & 1800 & 3 & 1143 & 1190 & TL & \textbf{1160} & \textbf{1168} & TL \\
200 & 1800 & 4 & 1217 & 1258 & TL & 1231 & 1231 & \textbf{279.6} \\
200 & 1800 & 5 & 1158 & 1204 & TL & \textbf{1168} & \textbf{1184} & TL \\
200 & 1800 & 6 & 1121 & 1179 & TL & \textbf{1143} & \textbf{1159} & TL \\
200 & 1800 & 7 & 1122 & 1168 & TL & \textbf{1128} & \textbf{1152} & TL \\
200 & 1800 & 8 & 1177 & 1239 & TL & \textbf{1203} & \textbf{1215} & TL \\
200 & 1800 & 9 & 1162 & 1212 & TL & \textbf{1180} & \textbf{1195} & TL \\
200 & 1800 & 10 & 1142 & 1177 & TL & \textbf{1149} & \textbf{1152} & TL \\
\bottomrule
\end{tabular}
\end{table}

For a better visual comparison, Figures \ref{fig:run} and \ref{fig:gap} show a plot of the runtime and the optimality gap of the computations of Tables~\ref{tab:comparison1} and~\ref{tab:comparison2}, respectively. This optimality gap is calculated as $(UB-best)/best \cdot 100$. 

The runtime plot confirms that our transformation-approach is much faster and allows to solve more instances. Furthermore, the optimality gap plot shows that it also produces significantly better bounds for instances that were not solved to optimality within the time limit.

\begin{figure}[h!tb]
	\centering
	\begin{subfigure}{0.45\textwidth}
		\includegraphics[width=\linewidth]{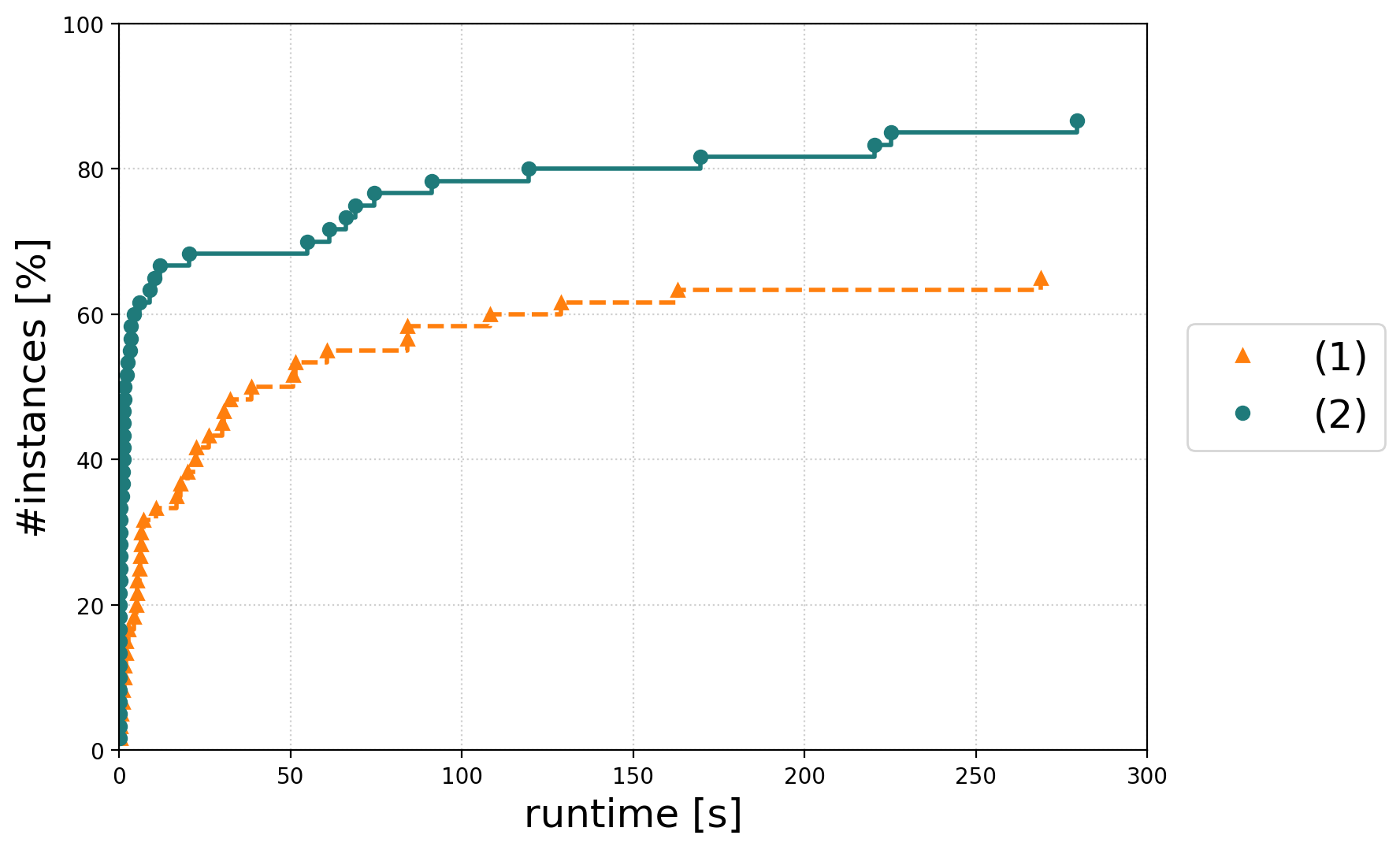}
		\caption{runtime}
		\label{fig:run}
	\end{subfigure}
	\begin{subfigure}{0.45\textwidth}
		\centering
		\includegraphics[width=\linewidth]{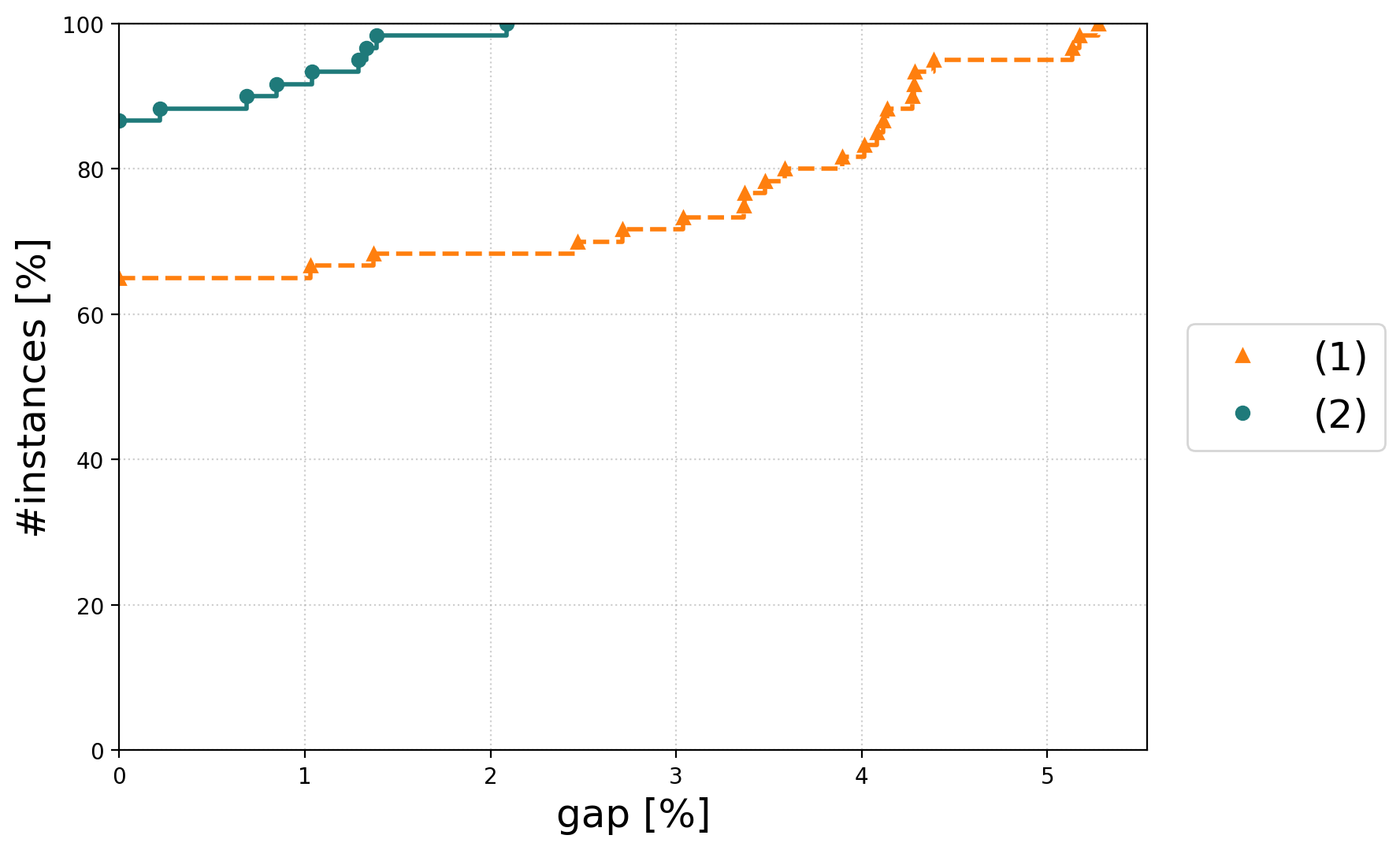}
		\caption{optimality gap}
		\label{fig:gap}
	\end{subfigure}
	\caption{Comparison of the formulations}
	\label{fig:plot}
\end{figure}

\section{Conclusions}
As our computational study demonstrated, to solve \BMCLPCP-instances it seems to be a good approach to simply transform the instances to \BMCLP-instances instead of developing specific algorithms for the \BMCLPCP. In particular, all approaches tailored for the \BMCLPCP need to be evaluated against those for the transformed instances utilizing the rich literature on the \BMCLP.
As a byproduct, our transformation reveals that the apparent complexity of the \BMCLPCP stems not from intrinsic difficulty, but from its representation.

\bibliographystyle{elsarticle-harv}
\bibliography{biblio}

\end{document}